\documentclass{amsart}

\usepackage{amsmath}
\usepackage{amssymb}
\usepackage{amsthm}

\newtheorem{theorem}{Theorem}[section]

\newtheorem{proposition}[theorem]{Proposition}
\newtheorem{lemma}[theorem]{Lemma}

\theoremstyle{definition}
\newtheorem{definition}[theorem]{Definition}
\newtheorem{remark}[theorem]{Remark}

\numberwithin{equation}{section}
\theoremstyle{definition}
\newtheorem*{theorem*}{Theorem}

\newcommand{\bbZ}{\mathbb{Z}}
\newcommand{\bbP}{\mathbb{P}}
\newcommand{\bbC}{\mathbb{C}}

\newcommand{\bbQ}{\mathbb{Q}}

\newcommand{\Oo}{\mathcal O}
\newcommand{\E}{\mathcal E}
\newcommand{\p}{\mathcal P}
\newcommand{\s}{\mathcal S}

\DeclareMathOperator{\diam}{diam}
\DeclareMathOperator{\dist}{dist}
\DeclareMathOperator{\PGL}{PGL}

\begin{document}

\title[Rational maps, bad reduction and domains of quasiperiodicity]{Rational maps with bad reduction and domains of quasiperiodicity}
\author{V\'ictor Nopal-Coello}
\address{Centro de Investigaci\'on en Matem\'aticas, A.P. 402 \\
C.P. 36240, Guanajuato, Gto., M\'{e}xico}
\email{victor.nopal@cimat.mx}

\author{M\'onica Moreno Rocha}
\email{mmoreno@cimat.mx}




\maketitle

\begin{abstract}
Consider a rational map $R$ of degree $d\geq 2$ with coefficients over the non-archimedean field $\bbC_p$, with $p$ a fixed prime number. If $R$ has a cycle of Siegel disks and has good reduction, then it was shown by Rivera-Letelier in \cite{RLT} that a new rational map $Q$ can be constructed from $R$, in such a way that $Q$ will exhibit a cycle of $m$-Herman rings. In this paper, we address the case of rational maps with bad reduction and provide an extension of Rivera-Letelier's result for these class of maps.
\end{abstract}

\section{Introduction}

Let $R\in \bbC(z)$ be a rational function of degree greater or equal than two. The \emph{Fatou set} of $R$, denoted by $F(R)$, is the largest open set where the iterates of $R$ form an equicontinuous family. A remarkable result of D.~Sullivan in \cite{Sul} states that each Fatou component must be eventually periodic. A collective endeavor spanning almost a century provided the classification of each periodic component into five types: attracting basins, super-attracting basins, parabolic basins, Siegel disks and Herman rings.

Siegel disks and Herman rings are also known as \emph{rotation domains} since an iterate of the map restricted to one of these periodic components is analytically conjugated to a rigid irrational rotation acting either on a disk or a ring domain, respectively. It was not an easy task to construct an example of a rational function with a cycle of Herman rings as illustrated by the works of M.~Herman in \cite{Herman} and M.~Shishikura in \cite{Sh}.

Consider now the problem of constructing a rational function that exhibits the analogue of a cycle of Herman rings in the $p$-adic setting. In~\cite{RLT}, J.~Rivera-Letelier proved that, if $R$ is a rational function with \emph{good reduction} and exhibits a cycle of Siegel disks, then it is possible to construct a new rational function of higher degree with a cycle of \emph{$m$-Herman rings} (see Theorem~\ref{thm:3} for the full statement of this result). Briefly, we say that $U$ is an $m$-Herman ring for $R$ if it belongs to its \emph{domain of quasiperiodicity} and for some integer $m\geq 1$, the set $U$ can be written as $\bbP(\bbC_p)\setminus (B_0\cup B_1\cup \ldots \cup B_m)$, where each $B_j$ represents a closed ball in the projective space of $\bbC_p$.

Our aim is to show that for a rational function with an $n$-cycle of Siegel disks and \emph{bad reduction}, it is still possible to produce a new rational function with an $n$-cycle of $1$-Herman rings (see Theorem~\ref{GRL}).

The organization of the paper is as follows: in Section 2 we provide a short introduction to $p$-adic analysis and state without proofs some of the most important results from rational iteration in the $p$-adic setting. In Section 3 we review several results known for rational maps with good reduction and their domains of quasiperiodicity. Then, we state and prove our main result for rational maps with bad reduction. In Section 4 we illustrate the conclusion of Theorem~\ref{GRL} with a couple of examples constructed from a rational map with bad reduction.

\section{Preliminaries}

\subsection{$p$-adic Analysis}

Let $p$ denote a fixed prime number. Since any rational number $r\in \bbQ$ can be written as $r=p^\alpha a/b$ for $\alpha, a,b\in \bbZ$ and $p$ not dividing either $a$ or $b$, then one can define the \emph{$p$-adic norm} as
$$|r|_p = p^{-\alpha},$$
with the convention $|0|_p=0$. It is not difficult to verify that $|\cdot|_p$ defines a norm whose triangle inequality is replaced by the \emph{non-archimedean} condition given by
$$|x+y|_p\leq \max\{|x|_p,|y|_p\}.$$

Denote by $\bbQ_p$ the completion of $\bbQ$ with respect to $|\cdot|_p$, by $\overline{\bbQ_p}$  the algebraic closure of $\bbQ_p$ and by $\bbC_p$ the completion of $\overline{\bbQ_p}$, which is algebraically closed. Throughout this work, we only consider the valued field $(\bbC_p,|\cdot|_p)$ and its projective space, $\bbP(\bbC_p)$, endowed with the \emph{non-archimedean chordal metric}, $\rho_p$. Its expression in homogeneous coordinates is given by 
$$\rho_p(P_1,P_2) = \frac{|X_1Y_2-X_2Y_1|_p}{\max\{|X_1|_p,|Y_1|_p\} \max\{|X_2|_p,|Y_2|_p\}},$$
where $P_i=[X_i,Y_i]\in \bbP(\bbC_p)$ for $i=1,2$. If $P_i\neq [1,0]$ for $i=1,2$, then under the change of coordinates $z_i=X_i/Y_i$, one obtains
$$\rho_p(z_1,z_2)=\frac{|z_1-z_2|_p}{\max\{|z_1|_p,1\} \max\{|z_2|_p,1\}}.$$

From now on, we drop the subscript $p$ form the $p$-adic norm and the chordal metric.
Consider
\begin{itemize}
\item the ring of integers $\Oo=\{z\in\bbC_p~|~|z|\leq 1\}$,
\item the group of units $\Oo^*=\{z\in \bbC_p~|~|z|=1\}$,
\item the maximal ideal of $\Oo$ given by $\mathfrak{p}=\{z\in \bbC_p~|~|z|<1\}$, and
\item the residue field of $\Oo$ given by $\tilde \bbC_p=\Oo/\mathfrak{p}$.
\end{itemize}
If $z\in \Oo$, its \emph{reduction modulo $\mathfrak{p}$} is denoted by $\tilde z$, and clearly, $\tilde z\in \tilde \bbC_p$. The mapping that takes $z\in \Oo$ to $\tilde z\in \tilde \bbC_p$ can be naturally extended to $\bbP(\bbC_p)$ by setting $\tilde z= \infty\in \bbP(\tilde \bbC_p)$ for all $z\in \bbP(\bbC_p)\setminus \Oo$.

The group $|\bbC_p|:=\{|z|~|~z\in \bbC_p\setminus \{0\}\}$ is called the \emph{valuation group} of $\bbC_p$.
Let $r\in |\bbC_p|$ and $z_0\in \bbC_p$. A \emph{ball} or a \emph{disk} in $\bbC_p$ centered at $z_0$ and radius $r$ is denoted as $B_r(z_0)=\{z\in \bbC_p~|~|z-z_0|\leq r\}$ and by $D_r(z_0)=\{z\in \bbC_p~|~|z-z_0|< r\}$, respectively. If $B$ is a ball in $\bbC_p$ then $\bbP(\bbC_p)\setminus B$ is a disk in the projective space, and viceversa.

Four immediate consequences from the non-archimedean condition are:
\begin{enumerate}
\item All triangles in $\bbC_p$ are isosceles: if $|x|>|y|$ then $|x+y|=|x|$.
\item If $U_0$ and $U_1$ are balls (or disks) in $\bbC_p$ with non-empty intersection, then either $U_0\subset U_1$ or $U_1\subset U_0$.
\item The radius of a ball (or a disk) in $\bbC_p$ is equal to its diameter.
\item Every point in a ball (or a disk) in $\bbP(\bbC_p)$ is its center.
\end{enumerate}

\subsubsection{Geometric objects}

The topology determined by the non-archimedean metric reduces every connected component of $\bbP(\bbC_p)$ to a single point. Hence, it is necessary to extend the idea of connected component in this setting. Most of the following concepts can be found in the dissertations by R.L. Benedetto in \cite{BenT}, and by J.~Rivera-Letelier in \cite{RLT} (alternatively, see also \cite{RL}).

\begin{definition}[Affinoids]
Let $U_1, U_2,\ldots, U_N\subset \bbP(\bbC_p)$ be a finite collection of balls (or disks). A \emph{closed (open) connected affinoid} is defined as $\bigcap_{j=1}^N U_j$. The finite union of closed (open) connected affinoids is a \emph{closed (open) affinoid}. The complement of a closed affinoid is an open affinoid.
\end{definition}

\begin{definition}[Analytic component]
Let $U\subset \bbP(\bbC_p)$ be a given subset and $x\in U$. The \emph{analytic component of $U$ containing $x$} is the union of all (closed or open) connected affinoids that contain $x$ and that are contained in $U$.
\end{definition}

An increasing union of connected affinoids defines a \emph{connected analytic space}. A disk is an example of a connected analytic space. Let $B_0\subset B_1\subset\ldots$ be an increasing sequence of balls in $\bbP(\bbC_p)$. Its increasing union, $B=\bigcup B_j$, is either a disk or it is equal to $\bbP(\bbC_p)$. The collection of open annuli (or disks if $B=\bbP(\bbC_p)$) given by $\{A_j~:A_j=B\setminus B_j~\}_{j\geq 0}$ is called a \emph{vanishing chain}. Two vanishing chains are equivalent if each chain is cofinal in the other under containment, see \cite[Chapter 10]{BR}.

\begin{definition}[End]
An equivalence class of vanishing chains, denoted by $\p$, is called an \emph{end}. The associated ball (resp. disk) of $\p$, denoted as $B_\p$ (resp. $D_\p$), is equal to $\bigcup B_j$ (resp. $\bbP(\bbC_p)-\bigcup B_j$) for any vanishing chain $\{A_j\}_{j\geq 0}$ and its definition is independent of the vanishing chain.
\end{definition}

Let $\p$ be an end associated to the \emph{canonical disk} $D_\p=\{z~|~|z|<1\}$. For each $w\in \bbP(\tilde \bbC_p)$, denote by $\p_w$ the end whose associated disk is
$$D_{\p_w}=\{z=w+\alpha~|~\alpha\in D_\p\},$$
thus, for any $z\in D_{\p_w}$, $\tilde z=w$. The set $\s_c=\{\p_w\}_{w\in \bbP(\tilde \bbC_p)}$ is called the \emph{canonical projective system} associated to $\p$ (or associated to the canonical disk $D_\p$). If $\p'$ is any other given end with disk $D_{\p'}$, its projective system $\s'$ is naturally defined under the action of a map $\psi \in \PGL(2,\Oo)$ so that $\psi(D_{\p'})=D_\p$.

\subsubsection{Rational maps and reduction}

A rational function $R\in \bbC_p(z)$ of degree $d\geq 2$ induces a rational map of same degree acting over the complex projective space. Written in homogeneous coordinates, the induced map $R:\bbP(\bbC_p)\to \bbP(\bbC_p)$ becomes $R(X,Y)=[F(X,Y),G(X,Y)]$, where $F,G\in \bbC_p[X,Y]$ are polynomials  in $X$ and $Y$ of degree $d$, namely
$$F(X,Y)= \sum_{k=0}^d a_k X^{d-k}Y^k,~~~G(X,Y)=\sum_{k=0}^d b_k X^{d-k} Y^k.$$
Whenever convenient, we shall write the rational map in $z$-coordinate as $R(z) = f(z)/g(z)$, where $F(X,Y)=Y^df(X/Y)$ and $G(X,Y)=Y^dg(X/Y)$. Also, we shall refer to a rational function $R\in \bbC_p(z)$ with the understanding that $R$ also represents its induced map over $\bbP(\bbC_p)$.

Following \cite{Sil}, we say that a pair of polynomials $(F,G)$ is \emph{normalized} if the coefficients of $F$ and $G$ lie in the ring of integers $\Oo$ and at least one coefficient of $F$ or $G$ is in $\Oo^*$. 

\begin{definition}[Rational map with good reduction]
Let $R=[F,G]$ where $(F,G)$ is a normalized pair. The \emph{reduction of $R$ mod $\mathfrak p$} is given by the expression
\begin{align*}
\tilde R(X,Y) &= [\tilde F(X,Y), \tilde G(X,Y)],\\
&=\left[ \sum_{k=0}^d \tilde a_k X^{d-k}Y^k,\sum_{k=0}^d \tilde b_k X^{d-k} Y^k\right].
\end{align*}
We say $R$ has \emph{good reduction} if $\deg(\tilde R)=\deg(R)$. If $R$ does not have good reduction, we say that $R$ has \emph{bad reduction}.
\end{definition}

\subsection{$p$-adic Iteration}
Let $R$ be a rational map of degree at least $2$. The expression $R^n$ denotes the $n$-fold composition of $R$ with itself $n$ times. A point $z_0\in \bbP(\bbC_p)$ is \emph{periodic of period $n$} if $R^n(z_0)=z_0$ and $n$ is the least positive integer that satisfies this condition. If $z_0\in \bbC_p$ is a periodic point of period $n$, its \emph{multiplier} is given by $\lambda:=(R^n)'(z_0)$. Multipliers are invariant under conjugacies in $\PGL(2,\Oo)$, \cite{Sil}. Thus, if $z_0$ is the point at infinity, a change of coordinates sending $z_0$ to a point in $\bbC_p$ allows us to compute its multiplier.

A periodic point with multiplier $\lambda$ is \emph{super-attracting, attracting, repelling} or \emph{indifferent} if $\lambda=0$, $|\lambda|<1$, $|\lambda| >1$ or $|\lambda|=1$, respectively.

Following \cite{Ben}, we define the \emph{Fatou set} of $R$, denoted by $F(R)$, as the set of points $z\in \bbP(\bbC_p)$ for which there exists an open neighborhood in $\bbP(\bbC_p)$ where $\{R^n\}_{n\geq 1}$ is an equicontinuous family with respect to the chordal distance $\rho$. The \emph{Julia set} of $R$, denoted by $J(R)$, coincides with $\bbP(\bbC_p)\setminus F(R)$.

The following are some of the properties  the Fatou and Julia sets satisfy in the p-adic setting,  proofs can be found in \cite{BenT}.

\begin{proposition}\label{prop:JF}
Let $R\in \bbC_p(z)$ be given. Then
\begin{enumerate}
\item $F(R)$ is open, $J(R)$ is closed and both sets are totally invariant under $R$. Moreover, $F(R^n)=F(R)$ and $J(R^n)=J(R)$ for all $n\geq 1$.
\item $F(R)$ is never empty. $J(R)$ has always empty interior.
\item A periodic point lies in $J(R)$ if and only if it is repelling.
\end{enumerate}
\end{proposition}

A remarkable implication for a rational map with good reduction is the following.

\begin{theorem}[Morton \& Silverman, \cite{MS}]\label{thm:MS}
If $R\in \bbC_p(z)$ has good reduction, then $R$ does not expand the chordal metric. Hence $J(R)=\emptyset$.
\end{theorem}

\subsubsection{Rational maps acting on geometric objects}

In order to study the dynamics of a rational map, it is essential to understand its action  on some of the geometric objects previously defined. Proofs of the following results can be found in Section 2 in \cite{RL} and Chapter 10 in \cite{BR}.

Let $\p$ be a fixed end with associated disk $D_\p$ and ball $B_\p$. Consider a vanishing chain representative $\{A_i\}_{i\geq 1}$ in the equivalence class $\p$ and let $\s=\{\p_w\}_{w\in \bbP(\tilde \bbC_p)}$ be the projective system associated to $\p$. We shall also refer to $\s$ as the projective system associated to $D_\p$ or to $B_\p$.

\begin{proposition}[Action on ends]\label{prop:RonEnds}
Let $R\in \bbC_p(z)$ of degree $d\geq 2$. Given $\p$, $D_\p$ and $\{A_i\}_{i\geq 1}$ as above,
\begin{enumerate}
\item there exists $m\geq 1$ such that, for any $\{A_i\}_{i\geq 1}$ and for $i$ sufficiently large, $R:A_j\to R(A_j)$ has degree $m$ for $j\geq i$.
In this case, the \emph{local degree} of $R$ in $\p$ is $\deg_R(\p)=m$.
\item $R(D_\p)=\bbP(\bbC_p)$ or $R:D_\p\to D_{R(\p)}$ with degree $\deg_R(\p)$.
\end{enumerate}
\end{proposition}

The following result states when two rational functions act exactly the same over ends. Its proof can be found in \cite[Lemma 2.3]{RL}.
 
\begin{proposition}\label{prop:RQ}
Consider two rational functions $R,Q\in \bbC_p(z)$ and an end $\p$ with vanishing chain representative $\{A_i\}_{i\geq 1}$.  Assume $B_{R(\p)}$ is of the form $\{z~|~|z-z_0|\leq r\}$. If  there exists $i_0\geq 1$ such that
$$|Q(z)-R(z)|\leq r~~~\text{for all}~~z\in A_{i_0},$$
then $Q(\p)=R(\p)$ and $\deg_Q(\p)=\deg_R(\p)$.
\end{proposition}

Let us consider the action of $R$ over projective systems. It was shown in \cite[Proposition 2.4]{RL} that given a projective system $\s_0$, there exists a projective system $\s_1$ so that if $\p\in \s_0$ then $R(\p)\in \s_1$.  In this case, we write $R(\s_0)=\s_1$. Now, consider two projective systems, $\s_0$ and $\s_1$, that satisfy $R(\s_0)=\s_1$. We can always find two automorphisms $\psi_0,\psi_1\in \PGL(2,\Oo)$ that send each projective system to the canonical projective system, that is $\psi_i(\s_i)=\s_c$ for $i=0,1$. Then, the rational function
$$R_*=\psi_1\circ R\circ \psi^{-1}_0$$
sends the canonical projective system into itself. By \cite[Corollary 9.27]{BR}, $R_*$ has a well-defined, non-constant reduction $\tilde R_*$. We say  $\tilde R_*$ is the \emph{reduction of $R$ with respect to $\psi_0$ and $\psi_1$}. The degree of $R$ in $\s_0$ is defined by $\deg_R(\s_0)=\deg(\tilde R_*)$.

A key result related to the action of $R$ over projective systems is stated in the next lemma. Given a connected analytic space $X$, we say $\s$ is \emph{contained} in $X$, denoted by $\s\prec X$, if $X$ is not contained in a disk associated to the projective system $\s$.

\begin{lemma}[Lemma 2.11 in \cite{RL}]\label{211}
If a rational function $R\in \bbC_p(z)$ is injective over the connected analytic space $X$, then for all projective systems $\s\prec X$, we have $\deg_R(\s)=1$.
\end{lemma}

\subsubsection{Fatou components}

In~\cite{RLT}, Rivera-Letelier introduced three types of $p$-adic Fatou components where the action of a rational map becomes relevant. These are the (immediate) \emph{basins of attraction},  \emph{wandering components}, and components of the \emph{domain of quasiperiodicity}. For completeness, we provide their definitions and list some of their elementary properties, further details can be found in \cite[Section 3]{RL}.

\begin{definition}[Basins of attraction]
Assume $z_0\in \bbP(\bbC_p)$ is a super-attracting or an attracting periodic point of $R$ of period $n\geq 1$. The \emph{basin of attraction} of $z_0$ is given by the set
$${\mathcal U}(z_0)=\{z\in \bbP(\bbC_p)~|~\rho(R^{nk}(z),R^{nk}(z_0))\to 0~\text{as}~k\to \infty\},$$
and the set $\bigcup_{m\geq 0} {\mathcal U}(R^m(z_0))$ is the \emph{basin of the cycle} of $z_0$.  The analytic component of ${\mathcal U}(z_0)$ that contains $z_0$ is called the \emph{immediate basin of attraction} of $z_0$.
\end{definition}

Standard arguments employed in rational dynamics over $\bbP(\bbC)$ show that every basin of an attracting cycle lies in the Fatou set of $R$.

\begin{definition}[Wandering disks]
A disk $D\subset \bbP(\bbC_p)$ is \emph{wandering} under $R$ if $R^n(D)\cap R^m(D)=\emptyset$ for all integers $n>m\geq 0$.
\end{definition}

In \cite[Lemma 4.29]{RL} it is shown that every wandering disks must lie in the Fatou set. Moreover, if $\diam_\rho$ denotes the diameter with respect to the chordal metric, then
$$\liminf_{n\to \infty} \diam_\rho( R^n(D) )=0.$$

\begin{definition}[Domain of quasiperiodicity]
Let $R\in \bbC_p(z)$ be a rational function. The \emph{domain of quasiperiodicity} of $R$ is defined as the set
$$\E(R)=\{z\in \bbC_p~|~\exists n_j\to \infty~\text{so that}~\rho(R^{n_j},\text{Id})\rightrightarrows 0~\text{in a neighborhood of}~z\},$$
where $\rightrightarrows$ denotes uniform convergence with respect to the chordal metric. 
\end{definition}

The following proposition gathers some of the main properties of $\E(R)$ found in \cite[Proposition 3.9]{RL}.

\begin{proposition}\label{prop:ER}
The domain of quasiperiodicity of a rational function $R$ satisfies the following properties:
\begin{enumerate}
\item $\E(R)$ is open, $R(\E(R))= \E(R)$ and $\E(R^n)=\E(R)$ for all $n>1$.
\item $R$ is injective over $\E(R)$.
\item If $h\in \PGL(2,\Oo)$ and $G=h\circ R\circ h^{-1}$, then $\E(G)=h(\E(R))$.
\end{enumerate}
\end{proposition}

Recall from Proposition~\ref{prop:RonEnds} that $R$ acts over a disk $D_\p$ by either sending it to all $\bbP(\bbC_p)$ or to $D_{R(\p)}$ with degree $\deg_R(\p)$. Lemma~4.21 in \cite{RL} establishes that, if $D_\p\subset\E(R)$ then $R:D_\p\to D_{R(\p)}$ acts with degree 1 and there exists $m\geq 1$ so that $R^m(D_\p)=D_\p$. In particular, $R^m|_{D_\p}$ does not expand the chordal metric, and thus $D_\p$ lies in $F(R^m)=F(R)$.

The final result of this section is interesting in its own right.

\begin{theorem}[Classification Theorem, \cite{RLT}]\label{thm:Classif}
The Fatou set of a rational function $R\in \bbC_p(z)$ consists of the following disjoint sets:
\begin{enumerate}
\item Basins of attraction.
\item $\bigcup_{n\geq 0} R^{-n}(\E(R))$.
\item The union of wandering disks that are not attracted by an attracting cycle.
\end{enumerate}
\end{theorem}

\section{Domains of Quasiperiodicity and  Reduction}\label{sec:EandRed}

In this section, we are solely interested in a rational function $R\in \bbC_p(z)$ acting on components of its domain of quasiperiodicity. The next result characterizes analytic components of $\E(R)$.

\begin{theorem}[Theorem 3 in \cite{RL}]\label{thm:3}
Let $R\in \bbC_p(z)$ be a rational function of degree $d\geq 2$ and let $C$ be an analytic component of $\E(R)$. Then $C$ is an open connected affinoid, that is,
\begin{equation}\label{eq:oca}
C=\bbP(\bbC_p) - \bigcup_{j=0}^m B_j,
\end{equation}
where $m\geq 0$ and each $B_j$ is a ball in $\bbP(\bbC_p)$. Moreover, each end $\p$ of $C$ (whose ball is associated to a  $B_j$) is periodic, and the degree of the induced action by the corresponding iterate of $R$ over the associated projective system is larger than one.
\end{theorem}

Following the terminology of Rivera-Letelier, we say that an open connected affinoid $C$ as in equation (\ref{eq:oca}) is a \emph{Siegel disk} if $m=0$; otherwise, $C$ is an \emph{$m$-Herman ring}. Thus, a $1$-Herman ring is an open connected affinoid of the form $\bbP(\bbC_p)-(B_0\cup B_1)$. 

\begin{remark}
It is worth noting that the dynamics of $R$ restricted to either a Siegel disk or an $n$-Herman ring in the $p$-adic setting is very different from the dynamics in $\bbP(\bbC)$, as each analytic component in $\E(R)$ contains an infinite number of indifferent periodic points (see \cite[Corollary 5.3]{RL}). Thus, a $p$-adic Siegel disk and a $p$-adic Herman ring do not longer correspond to the largest domains of linearization, as expected for rational dynamics over $\bbP(\bbC)$.
\end{remark}

For further reference, we state the following result regarding the existence of fixed points in the interior of a Siegel disk.

\begin{proposition}[Corollary 5.11 in \cite{RL}]\label{511}
If $D$ is a fixed Siegel disk, then $D$ contains at least one fixed point.
\end{proposition}

\subsection{Good reduction}

Corollary 4.34 in \cite{RL} states that if $R\in \bbC_p(z)$ has good reduction and $\E(R)\neq \emptyset$, then every analytic component $C\subset \E(R)$ is a Siegel disk. In analogy to rational dynamics in $\bbP(\bbC)$, it is then natural to consider in our setting how to construct a rational function with an $n$-cycle of Herman rings starting from a rational function with an $n$-cycle of Siegel disks.~In \cite{Sh}, M.~Shishikura provided an elegant solution to this problem using quasiconformal surgery. The following result by Rivera-Letelier provides a clever answer in the $p$-adic setting without applying quasiconformal techniques.

\begin{theorem}[Proposition 6.7 in \cite{RL}]\label{thm:GR}
Let $R\in \bbC_p(z)$ with good reduction, $r\in |\bbC_p|$ so that $r<1$, and let $\p_1, \ldots, \p_n$ be ends so that each $B_{\p_i}\subset \E(R)$ has radius $r$. Then, there exists $Q\in \bbC_p(z)$ with $\deg(Q)=\deg(R)+n$, so that for every analytic component $C$ of $\E(R)$, the open connected affinoid
$$C-\cup_{T}B_{\p},$$
with $T=\{R^j(\p_i)~|~1\leq i\leq n, j\geq 1\}$, is an analytic component of $\E(Q)$.
\end{theorem}

Since $C$ is a Siegel disk, the resulting analytic component for $Q$ must be an $m$-Herman ring for some finite $m\geq 1$. Moreover, $Q$ must have bad reduction. The proof of the above theorem is based on Proposition~\ref{prop:RQ} and the following result, which we state for further reference.

\begin{proposition}[Proposition 5.2 in \cite{RL}]\label{prop:P5.2}
Let $X$ be an open connected affinoid and $R\in \bbC_p(z)$ be a rational map of degree at least 2 such that $R:X\to X$ has degree 1. Then $X\subset\E(R)$. Moreover, if for every projective system $\s$ there exists $n\geq 1$ so that $\deg_{R^n}(\s)>1$, then $X$ is an analytic component of $\E(R)$.
\end{proposition}

\subsection{Bad reduction}

The hypothesis of good reduction in Theorem~\ref{thm:GR} permits the selection of ends whose associated balls have all the same radii. This is not necessarily true for maps with bad reduction. Nevertheless, our first result states that, regardless of the reduction of $R$, we can retain control on the radius of $R(U)$ for a given disk (or ball) $U\subset \E(R)$. 

\begin{lemma}\label{lem:rad}
Let $R\in\mathbb{C}_p(z)$ be a rational function with $\mathcal{E}(R)\neq\emptyset$. Let $D\subset\mathcal{E}(R)$ be a disk of radius $r$ and $R(D)$ a disk of radius $t$. Then
$$|R(z_1)-R(z_2)|=\frac{t}{r}|z_1-z_2|,$$
for all $z_1,z_2\in D$. 
\end{lemma}

\begin{proof}
Let $z_1,z_2\in D$ and consider the function $f(z):=R(z)-R(z_2)$, which sends the disk $D_r(z_2)$ to $D_t(0)$. Proposition~\ref{prop:ER} establishes that $R$ is injective in $D=D_r(z_2)$, hence, $f$ has only one zero in $D_r(z_2)$. Moreover, we can find values $\alpha_i,\beta_j\notin D_r(z_2)$ for $1\leq i\leq n,1\leq j\leq m$ so that
\begin{equation}\label{eq:factor}
f(z)=a(z-z_2)\frac{(z-\alpha_1)\cdots(z-\alpha_n)}{(z-\beta_1)\cdots(z-\beta_m)},
\end{equation}
for some $a\in \bbC_p$.
Since $|z-z_2|<r\leq|z_2-\alpha_i|,|z_2-\beta_i|$ for each $i,j$, we obtain
\begin{align*}
|f(z)|&=|a(z-z_2)|\frac{|(z-\alpha_1)\cdots(z-\alpha_n)|}{|(z-\beta_1)\cdots(z-\beta_m)|},\\
&=|a(z-z_2)|\frac{|(z-z_2+z_2-\alpha_1)\cdots(z-z_2+z_2-\alpha_n)|}{|(z-z_2+z_2-\beta_1)\cdots(z-z_2+z_2-\beta_m)|},\\
      &=|a(z-z_2)|\frac{|(z_2-\alpha_1)\cdots(z_2-\alpha_n)|}{|(z_2-\beta_1)\cdots(z_2-\beta_m)|},
\end{align*}
where the last expression follows from the isosceles condition of the non-archimedean norm. The disk $R(D)$ has radius $t$, so $|f(z)|\rightarrow t$ when $|z-z_2|\rightarrow r$. This implies that
 $$\frac{t}{r}=|a|\frac{|(z_2-\alpha_1)\cdots(z_2-\alpha_n)|}{|(z_2-\beta_1)\cdots(z_2-\beta_m)|},$$
and together with (\ref{eq:factor}), we  conclude that $|R(z_1)-R(z_2)|=|f(z_1)|=\frac{t}{r}|z_1-z_2|$.
 
\end{proof}

\begin{remark}
Proposition~3.14 in \cite{BenT} provides a similar result for power series defined over a general non-archimedean field.
\end{remark}

We now state and prove our main result.

\begin{theorem}\label{GRL}
Let $R\in\mathbb{C}_p(z)$ be a rational function with an $n$-cycle of Siegel disks, $D_0,D_1,\ldots,D_{n-1}\subset \E(R)$, labeled in such a way that if $\rho_j$ is the radius of $D_j=R^j(D_0)$, then $\rho_0\leq\rho_j$ for $1\leq j\leq n-1$.

If the distance between $D_0$ and $D_j$ is the same as the distance between $D_1$ and $D_{j+1}$ for each $j\geq 1$, then there exists a rational function $Q\in\mathbb{C}_p(z)$ with $\deg(Q)=\deg(R)+1$ such that $Q$ has an $n$-cycle of 1-Herman rings.    
\end{theorem}

\begin{proof}
After a change of coordinates, assume that $D_j\subset B_1(0)$ for $0\leq j\leq n-1$, so each Siegel disk is a disk in $\bbC_p$. From Proposition~\ref{511}, we can assume the existence of an $n$-periodic point $z_0\in D_0$. Let $r\in|\mathbb{C}_p|$ and $\mu\in\mathbb{C}_p$ so that $|\mu|=r<\rho_0$, and define
\begin{equation}\label{eq:Q}
Q(z)=\frac{z-z_0}{z-z_0-\mu}(R(z)-R(z_0))+R(z_0).
\end{equation}

Our goal is to show the collection $\{D_j-R^j(B_r(z_0))\}_{j=0}^{n-1}$ is an $n$-cycle of analytic components in $\E(Q)$, and thus, they define an $n$-cycle of 1-Herman rings for $Q$.

We start by showing that $Q$ acts injectively on $D_j-R^j(B_r(z_0))$ for each $0\leq j\leq n-1$. We start with the case $j=0$: let $\p$ be an end whose associated ball, $B_\p$, lies in $D_0-B_r(z_0)$ and has radius $r<\rho_0$. For any pair $z_1, z_2\in B_\p\subset D_0$, we can apply Lemma~\ref{lem:rad} to $D_0$ and $D_1$ to obtain
$$|R(z_1)-R(z_2)|= \frac{\rho_1}{\rho_0}|z_1-z_2|\leq \frac{\rho_1}{\rho_0}r.$$
Since $z_1$ and $z_2$ are arbitrary, the radius of $B_{R(\p)}$ must be $\frac{\rho_1}{\rho_0}r$. Now consider an arbitrary point $z\in B_\mathcal{P}$ and observe that $z,z_0\in D_0$, hence, their images lie in $D_1$. In order to show $Q$ and $R$ act exactly the same over ends, we estimate
\begin{align*}
|Q(z)-R(z)|&=\left|\frac{z-z_0}{z-z_0-\mu}(R(z)-R(z_0))+R(z_0)-R(z)\right|, \\
           &=\frac{|\mu(R(z)-R(z_0))|}{|z-z_0-\mu|}=|\mu|\frac{|R(z)-R(z_0)|}{|z-z_0|},
\end{align*}
where the last identity follows after applying the isosceles property in the denominator. Since $z\in B_\p$ is arbitrary, the non-archimedean condition implies that  $|z-z_0|=\rho_0$ while $|R(z)-R(z_0)|=\rho_1$. Then
$$|Q(z)-R(z)|=r \frac{\rho_1}{\rho_0},$$
and Proposition~\ref{prop:RQ} implies that $Q(\mathcal{P})=R(\mathcal{P})$ and $\deg_Q(\mathcal{P})$ $= \deg_R(\mathcal{P})=1$. That is, $Q$ acts injectively on $D_0-B_r(z_0)$.

Now consider any $1\leq j\leq n-1$. Let $\p_j$ be an end whose associated ball, $B_{\mathcal{P}_j}$, lies in $D_j-R^j(B_r(z_0))$ and has radius $\frac{\rho_j}{\rho_0}r<\rho_j$. As before, selecting two arbitrary points in $B_{\p_j}\subset D_j$ and applying Lemma~\ref{lem:rad} to $D_j$ and $D_{j+1}$, one concludes that $B_{R(\mathcal{P}_j)}$ has radius $\frac{\rho_{j+1}}{\rho_0}r$ (when $j=n-1$, the subscript $j+1$ is taken mod $n$).

Given an arbitrary point $z\in B_{\mathcal{P}_j}$, we now observe that $z_0\in D_0$ while $z\in D_j$. By the distance assumption on the Siegel disks, namely $\dist(D_0,D_j)=\dist(D_1,D_{j+1})$, it follows that $|z-z_0|=|R(z)-R(z_0)|$. Then, the estimate becomes
\begin{align*}
|Q(z)-R(z)|&=|\mu|\frac{|R(z)-R(z_0)|}{|z-z_0|}=r< \frac{\rho_{j+1}}{\rho_0}r.
\end{align*}
Therefore $Q(\mathcal{P}_j)=R(\mathcal{P}_j)$ and $\deg_Q(\mathcal{P}_j)=\deg_R(\mathcal{P}_j)=1$. Once more, we conclude that $Q$ acts injectively on $D_j-R^j(B_r(z_0))$ for any $1\leq j\leq n-1$.

The first part of Proposition~\ref{prop:P5.2} implies that each $D_j-R^j(B_r(z_0))$ lies in $\E(Q)$. In order to show each ring is an analytic component, we need to prove that $Q$ acts on projective systems with degree larger than one.

Since the disk $D_r(z_0)$ lies in $\E(R)$, then $R$ acts injectively on it. Thus, for any projective system $\s_0$ associated to $D_r(z_0)$, Lemma~\ref{211} shows that $\deg_R(\s_0)=1$. Consider the Mobius transformations $\psi_r,\psi_t\in \PGL(2,\Oo)$, where $\psi_r$ sends $D_r(z_0)$ to $D_1(0)$ while $\psi_t$ sends $R(D_r(z_0))=D_t(R(z_0))$ to $D_1(0)$. Recall that condition $\deg_R(\s_0)=1$ is equivalent to state that $\deg(\tilde R_*)=1$, where $R_*=\psi_t\circ R\circ \psi_r^{-1}$ preserves the canonical projective system $\s_c$ associated to the canonical disk $D_1(0)$.

In order to compute $Q_*$ and the degree of its reduction, first observe the image of $D_r(z_0)$ under both $R$ and $Q$ is a disk of radius $t$. Indeed, expression (\ref{eq:Q}) implies that for any $z\in D_r(z_0)$,
$$|Q(z)-Q(z_0)|= \frac{|z-z_0||R(z)-R(z_0)|}{|z-z_0-\mu|} = \frac{|z-z_0|}{r}|R(z)-R(z_0)|.$$
As $|z-z_0|\to r$, we obtain $|Q(z)-Q(z_0)|\to |R(z)-R(z_0)|=t$ and thus $Q(D_r(z_0))=D_t(Q(z_0))$. Thus, we define $Q_*=\psi_t\circ Q\circ \psi_r^{-1}$, and since $Q(z_0)=R(z_0)$, the expression of $Q_*$ becomes
\begin{align*}
Q_*(z) = &\frac{Q(rz+z_0)-Q(z_0)}{t},\\
          =&\frac{rz}{rz-\mu}\left( \frac{R(rz+z_0)-R(z_0)}{t}\right), \\
          =&\left(\frac{z}{z-\mu/r}\right) R_*(z).
\end{align*}

The coefficients in parenthesis of the last identity belong to $\Oo^*$ and $\deg(\tilde R_*)=1$. Thus, $\tilde Q_*$ is a rational function of degree 2, and in particular, $\deg_Q(\s_0)>1$. As every end associated to $R^j(D_r(z_0))$ is periodic under $Q$ with period $n$, then these ends are fixed under $Q^n$. Since $Q$ is injective in $D_0-B_r(z_0)$ and acts $2$-to-$1$ over the ends of the projective systems, there exist ends $\mathcal{P}_0,\mathcal{P}_1\in \s_0$ such that $Q(\mathcal{P}_0)=Q(\mathcal{P}_1)$. Similarly, if $\s_j$ denotes a projective system associated to $R^j(D_r(z_0))$ for each $0< j\leq n-1$, there exist $\mathcal{P}_{0,j},\mathcal{P}_{1,j}\in \s_j$ such that $Q^n(\mathcal{P}_{0,j})=Q^n(\mathcal{P}_{1,j})$, and therefore $\deg_{Q^n}(\s_j)>1$. From Proposition~\ref{prop:P5.2}, each $D_j-R^j(B_r(z_0))$ is an analytical component of $\E(Q^n)$ and hence, it is an analytical component of $\mathcal{E}(Q)$ by Proposition~\ref{prop:ER}. We conclude from Theorem~\ref{thm:3} that the collection
$$\{D_j-R^j(B_r(z_0))\}_{j=0}^{n-1}$$
is an $n$-cycle of $1$-Herman rings for $Q$.

\end{proof}

\section{Examples}

We provide a couple of examples that realize the setting described in Theorem~\ref{GRL}. The first example considers a rational map $R$ with a $2$-cycle of Siegel disks $\{D_0, D_1\}$ and a point $z_0\in D_0$ which is $2$-periodic under $R$. In the second example, we show that the periodicity condition imposed to $z_0$ can be relaxed: indeed, we construct $Q$ by selecting a non-periodic point in $D_0$ whose second iterate comes back to $D_0$ without being periodic.

\subsection*{General setting}
Let $p=5$ and consider $R\in\mathbb{C}_5(z)$ defined by
\begin{equation}\label{eq:r}
R(z)=\frac{z^2-z/5}{z^2-1},
\end{equation}
or in homogeneous coordinates
$$R[X,Y]=[5X^2-XY,5X^2-5Y^2],$$
with normalized polynomials. The reduction of $R$ becomes $\tilde R(X,Y)=[-XY,0]$, that is, $R$ has bad reduction. A direct computation shows that the origin is a  fixed point of $R$ with multiplier $\lambda=1/5$ and thus, repelling under the $p$-adic norm, as $|1/5|=5$. From Proposition~\ref{prop:JF}, it follows that $J(R)$ is not empty.

A straightforward computation shows that $R$ also has an indifferent 2-cycle at $\{1,\infty\}$. Let us show the disks $D_1(1)$ and $D_5(\infty):=\{z~|~|z|>5\}\cup\{\infty\}$ form a 2-cycle of Siegel disks.

For any $z\in D_1(1)$, it can be written as $z=1+\alpha$, with $|\alpha|<1$. Then
$$|R(z)| =\frac{|z||z-1/5|}{|z-1||z+1|} = \frac{|1+\alpha+1/5|}{|\alpha| |2+\alpha|} = \frac{5}{|\alpha|}>5,
$$
so in particular $R$ sends $D_1(1)$ into $D_5(\infty)$. A similar computation shows that if $z\in D_5(\infty)$, then $|R(z)-1|=5/|z|<1$, so $D_5(\infty)$ maps into $D_1(1)$. We claim $R^2(D_1(1))=D_1(1)$. Indeed, for $z=1+\alpha$, $|\alpha|<1$,
\begin{align*}
|R^2(z)-1|&=\left|\frac{(R(z))^2-R(z)/5}{(R(z))^2-1}-1\right|=\left|\frac{-R(z)/5+1}{(R(z))^2-1}\right|=\left|\frac{-R(z)/5}{(R(z))^2}\right|,\\
          &=5\left|\frac{R(z)}{(R(z))^2}\right|=5\left|\frac{1}{R(z)}\right|=5\left|\frac{z^2-1}{z^2-z/5}\right|,\\
          &=5\left|\frac{(1+\alpha)^2-1}{(1+\alpha)^2-\frac{1+\alpha}{5}}\right|=5|\alpha|\left|\frac{2+\alpha}{(1+\alpha)^2+1/5+\alpha/5}\right|=|\alpha|.
\end{align*}
This computation shows that $R^2$ does not expand or contract the norm in $D_1(1)$, hence, the collection $\{D_1(1),D_5(\infty)\}$ is a $2$-cycle of Siegel disks for $R$. Under the conjugacy of the automorphism $\varphi(z)=1/z$ we obtain
\begin{equation}\label{eq:rfi}
R^{\varphi}(z)=\varphi \circ R\circ \varphi^{-1}(z)=\frac{5-5z^2}{5-z}
\end{equation}
has a $2$-cycle of Siegel disks $\{D_{1/5}(0),D_1(1)\}$ in $\bbC_p$. Moreover, $R^\varphi$ also has bad reduction. Both $R^\varphi$ and its cycle of Siegel disks satisfy the hypothesis of Theorem~\ref{GRL}, so we can proceed to construct $Q\in \bbC_p(z)$ with $\deg(Q)=3$ that exhibits a $2$-cycle of 1-Herman rings.

\subsection*{Example 1}
Set  $D_0:=D_{1/5}(0), D_1:=D_1(1)$, let $\mu=25, r=1/25$ and $z_0=0$, which is a $2$-periodic point that lies in $D_0$. The expression in (\ref{eq:Q}) applied to $R^\varphi$ defines the new rational function
\begin{align*}
Q(z)&=\frac{z}{z-25}(R^\varphi(z)-1)+1,\\
&= \frac{5z^3-30z+125}{(z-25)(z-5)}.
\end{align*}

Let us compute the image of the ball $B_{1/25}(0)$ under $R^\varphi$. First, observe that $z_1=125\in D_{1/25}(0)$. Set $z_2=0$. Since $R^\varphi(125)=651$ and $R^\varphi(0)=1$, Lemma~\ref{lem:rad}  implies that
\begin{align*}
\frac{1}{25}=|650|=|R^\varphi(125)-R^\varphi(0)| = \frac{t}{1/25}|125|= \frac{t}{5}.
\end{align*}
Thus $t=1/5$ and $R^\varphi$ maps $B_{1/25}(0)$ to $B_{1/5}(1)$. Now, consider the rings
$$A_{1/25}^{1/5}(0)=\{z\in\mathbb{C}_p~|~1/25<|z|<1/5\}=D_0-B_{1/25}(0),$$
and
$$A_{1/5}^1(1)=\{z\in\mathbb{C}_p~|~1/5<|z-1|<1\}=D_1-B_{1/5}(1).$$

Theorem~\ref{GRL} implies that these rings form a 2-cycle under $Q$ and this map does not expand or contract the metric, as expected. Observe also that $A_{1/25}^{1/5}(0)$ (and hence $A_{1/5}^1(1)$) are maximal analytic components of quasiperiodicity: if we consider a ring domain $A_\rho^{1/5}(0)$ with $\rho<1/25$, then the pole at $\mu$ must lie in its interior, so $Q(A_\rho^{1/5}(0))$ must contain a disk $D_\delta(\infty)\subset \bbP(\bbC_p)$ for some $\delta<1$. In particular, $Q$ will not longer be injective in $A_\rho^{1/5}(0)$, a contradiction with part (2) in Proposition~\ref{prop:ER}.

\subsection*{Example 2}

Consider again $R^\varphi$ as in (\ref{eq:rfi}) and $p=5$. Here we stress that the periodicity condition imposed on $z_0$ can be relaxed: in order to construct $Q$, it is enough to select a non-periodic point in a Siegel disk whose second iterate returns to the same disk. 

Recall that the non-archimedean condition implies that any point in a disk is its center. Thus, we now work with the disks
$$D_0:=D_{1/25}(125)=D_{1/25}(0)~~~\text{ and }~~~D_1:=D_{1/5}(651)=D_{1/5}(1).$$
As $D_0\subset D_{1/5}(0)$ and $D_1\subset D_1(1)$, it is clear that $R^\varphi$ acts injectively on these disks. And from the previous example, $R^\varphi$ maps $D_0$ to $D_1$, thus $\{D_0, D_1\}$ forms a new $2$-cycle of Siegel disks for $R^\varphi$. 
Therefore, we can apply Theorem~\ref{GRL} and define a rational function $Q$ with $\mu=25$, $r=1/25$ and $z_0=125$. Observe that in this case, $z_0$ is not longer $2$-periodic since  $R^\varphi(125)=615$ and $R^\varphi(615)=3280\frac{60}{323} \neq 125$.
The new rational map is given by
\begin{align*}
Q(z)&=\frac{z-125}{z-150}(R^\varphi(z)-651)+651, \\
&=\frac{5z^3-1276z^2+83974z-308600}{150-z}.
\end{align*}

Again, by Theorem~\ref{GRL}, the rings $A_{1/25}^{1/5}(125)=A_{1/25}^{1/5}(0)$ and $A_{1/5}^1(651)=A_{1/5}^1(1)$ form a $2$-cycle of 1-Herman rings for our new map $Q$.

\end{document}